\documentclass[12pt,a4paper]{article}
\usepackage{amssymb,amsmath,amsthm}
\usepackage{graphicx}
\usepackage{hyperref}

\usepackage{authblk}
\newcommand\cF{{\mathcal F}}

\newcommand{\abs}[1]{\left\lvert{#1}\right\rvert}
\newcommand{\floor}[1]{\left\lfloor{#1}\right\rfloor}

\theoremstyle{plain}
\newtheorem{theorem}{Theorem}[section]
\newtheorem{lemma}[theorem]{Lemma}

\newtheorem{conjecture}[theorem]{Conjecture}
\newtheorem{proposition}[theorem]{Proposition}

\theoremstyle{definition}

\newtheorem{defn}[theorem]{Definition}

\newcommand\cref[1]{Corollary~\ref{cor:#1}}

\title{Set systems related to a house allocation problem}
\date{}

\begin{document}
% uncomment below for authblk version
\title{Set systems related to a house allocation problem}

\author[1]{D\'aniel Gerbner} 
\author[1,2]{Bal\'azs Keszegh}
\author[4]{Abhishek Methuku}
\author[1]{D\'aniel T. Nagy}
\author[1]{Bal\'azs Patk\'os}
\author[3,4]{Casey Tompkins} 
\author[5]{Chuanqi Xiao} 

\affil[1]{\small Alfr\'ed R\'enyi Institute of Mathematics, Hungary}
\affil[2]{\small MTA-ELTE Lend\"ulet Combinatorial Geometry Research Group, Hungary}
\affil[3]{\small Karlsruhe Institute of Technology, Germany}
\affil[4]{\small Discrete Mathematics Group, Institute for Basic Science (IBS), Daejeon, Republic of Korea}
\affil[5]{\small Central European University, Hungary}

\maketitle

\begin{abstract}
%change "is strictly better than" to dominates? -ct
We are given a set $A$ of buyers, a set $B$ of houses, and for each buyer a preference list, i.e., an ordering of the houses. A house allocation is an injective mapping $\tau$ from $A$ to $B$, and $\tau$ is strictly better than another house allocation $\tau'\neq \tau$ if for every buyer $i$, $\tau'(i)$ does not come before $\tau(i)$ in the preference list of $i$. A house allocation is Pareto optimal if there is no strictly better house allocation.

Let $s(\tau)$ be the image of $\tau$ (i.e., the set of houses sold in the house allocation $\tau$). We are interested in the largest possible cardinality $f(m)$ of the family of sets $s(\tau)$ for Pareto optimal mappings $\tau$ taken over all sets of preference lists of $m$ buyers. We improve the earlier upper bound on $f(m)$ given by Asinowski, Keszegh and Miltzow~\cite{AKM}, by making a connection between this problem and some problems in extremal set theory.
    
\end{abstract}
\section{Introduction}

%\begin{theorem}The number of POMs is at most ...\end{theorem}

In this short note we consider set system properties related to the following house allocation problem: Suppose we are given an $m$-element set $A$ of buyers, an infinite set $B$ of houses and for each buyer $a \in A$ we are given a preference list of $B$. Then we say that a matching $\tau$ from $A$ to $B$ is \emph{Pareto optimal} (a POM) if there does not exist a nonempty subset $A'$ of $A$ (a \emph  {blocking coalition}) and another matching $\tau'$ from $A$ to $B$ that differs from $\tau$ only on $A'$ such that for every $a\in A'$, $\tau'(a)$ is placed higher in the preference list of $a$ than $\tau(a)$. If for some matching $\tau$, there does not exist a blocking coalition of size at most $i$, we say that $\tau$ is an $i$-POM.

%Changed all the POM's to POMs because plural of POM is POMs. -Abhi
There is a very natural way to generate POMs: Consider a permutation $\pi$ of $A$ and let the buyers select their favorite house from those that are not already taken, in order $\pi(1), \pi(2),\dots,\pi(m)$. The corresponding matching is denoted by $\tau_\pi$. For any matching $\tau$ the set $\{\tau(a):a\in A\}$ of houses which are taken is denoted by $s(\tau)$. The following lemma from~\cite{AKM} summarizes all that we will use in this paper about POMs. We say that a subset $E$ of $B$ is reachable if there exists a POM $\tau$ with $s(\tau)=E$.

\begin{lemma}[Lemma 7 in~\cite{AKM}]\label{equiv}
Let $E\subseteq \{ 1,...,n\}$ with $|E|=m$. The following statements are equivalent.
\begin{enumerate}
    \item 
    $E$ is reachable, i.e., there exists a POM $\tau$ with $s(\tau) =E$.
    \item
    There exists a permutation $\pi$ such that for the greedy matching $\tau_\pi$ we have $s(\tau_\pi) =E$.
    \item
    There exists a 1-Pareto optimal matching (1-POM) $\tau$ with $s(\tau) =E$.

\end{enumerate}
\end{lemma}

 The number of reachable sets depends on the preference lists. Let $f(m)$ denote the maximum number of sets of reachable houses taken over all $m$-tuples of preference lists. We will elaborate more on the motivation and background of the problem of determining $f(m)$ in the next subsection.

%geometric motivation \cite{BHJKRST}.

Let us begin by introducing the set system properties that we will be interested in.

\begin{defn}
We say that a family $\cF$ of $m$-element sets has property $P$ if for any integer $k$ and any $F_1,\dots, F_k\in\cF,$ there exists $j$, $1 \le j \le k,$ such that $\abs{F_j \setminus \cup_{i \neq j} F_i} \le \floor{m/k}$. Let $g(m)$ denote the largest cardinality of a family with property $P$.
\end{defn}

\begin{defn}
We say that a family $\cF$ of $m$-element sets has property $Q$ if for any integer $k$ and any $F_1,\dots, F_k\in\cF$, we have $|\bigcup_{i=1}^k F_i|\le \sum_{i=1}^k \floor{m/i}$. Let $h(m)$ denote the largest cardinality of a family with property $Q$.
\end{defn}

Obviously property $P$ implies property $Q$. Therefore, $g(m)\le h(m)$.
The following lemma is hidden in the proof of the upper bound on the number of reachable houses from~\cite{AKM}.
\begin{lemma}[\cite{AKM}]
\label{reach} The family of the sets of reachable houses has property $P$ and thus $f(m)\le g(m)$.
%Let $F_1,F_2,\dots,F_k$ be reachable sets of houses.  Then there exists $j$, $1 \le j \le k$ such that $\abs{F_j \setminus \cup_{i \neq j} F_i} \le \floor{m/k}$.
\end{lemma}

The upper bound in~\cite{AKM} on the number of houses uses only property $Q$, showing that if $\cF$ has property $Q$, then $|\bigcup\cF|\le m(\ln m+1)$. This obviously implies $|\cF|\le \binom{m(\ln m+1)}{m}$ and thus  $f(m)\le g(m)\le h(m)\le |\cF|\le \binom{m(\ln m+1)}{m}$. We improve this bound on $|\cF|$ as follows.

\begin{theorem}\label{main} 
%The number of reachable sets of houses is at most $\binom{2m}{m}$. 
(i) Let $\ell=\ell(m)$ be the largest integer with $\ell^{\ell}\le m$. Then $f(m)\le \binom{2m-\ell+4}{m}$ holds if $m$ is sufficiently large. In other words, the family of the sets of reachable houses taken over all $m$-tuples of preference lists contains at most $\binom{2m-\ell+4}{m}$ sets if $m$ is sufficiently large.

(ii) We have $g(m)\le \binom{2m-1}{m}$. In other words, if a family $\cF$ of $m$-element sets has property~$P$, then $|\cF|\le \binom{2m-1}{m}$.
\end{theorem}

Note that $\ell\ge \frac{\log m}{\log\log m}$, therefore the bounds of (i) and (ii) differ by a factor of $\frac{\log m}{m}$.
%This implies that $f(m)\le \binom{2m}{m}$, and 

%The first sentence below is still quite weird, but I don't see what you want to say exactly. -ct
\subsection{Applications to pattern matching problems}
In the topic of matchings the problem of counting the number of sets of reachable houses has still not received considerable attention. The problem was first raised in~\cite{HJK} (a precursor of~\cite{BHJKRST}) where it was shown that it can be applied to upper bound the complexity of a certain subdivision of the plane which arises from a pattern matching problem. Moreover, this was used to derive efficient algorithms for this pattern matching problem. 

Before stating the connection and the implication of our result, we repeat the most important definitions from  \cite{BHJKRST}.

Let $A=\{a_1,\ldots,a_n\}$ and $B=\{b_1,\ldots,b_m\}$ be two sets of points in the plane. We assume that $m \le n$, and we seek a minimum-weight maximum-cardinality matching of $B$ into $A$. This is a subset $M$ of edges of the complete bipartite graph with edge set $B\times A$, so that each $b\in B$ appears in exactly one edge of $M$, and each $a\in A$ appears in at most one edge. The weight of an edge $(b,a)$ is $\|b-a\|^2$, the squared Euclidean distance between~$b$ and $a$, and the weight of a matching is the sum of the weights of its edges.

Allowing the pattern $B$ to be translated, we obtain the problem of computing the minimum partial-matching RMS (Root Mean Square) distance under translation. That is,  we want to find a translation vector $t$ and a matching of $B$ into $A$ such that the weight of the matching edges determined by the Euclidean distances between $B(t)$ ($B$ translated by $t$) and $A$ is minimal.

This problem induces a subdivision of the plane, where two points $t_1,t_2$ are in the same region if when translating $B$ by $t_1$ or $t_2$, the above defined minimum-weight matching of $B$ to $A$ is the same. This is called the \emph{partial-matching subdivision} and is denoted by $\cal D_{B,A}$.

In \cite{HJK,BHJKRST} it was shown that the combinatorial complexity of $\cal D_{B,A}$ can be upper bounded by $O(n^2m^4f(m))$ and that computing a global minimum of the partial-patching RMS distance can be done in time $O(n^3m^8f(m))$.

Using the result of \cite{AKM} this implied for them the upper bound $O(n^2m^{3.5} (e \ln m+e)^m)$ for the complexity of $\cal D_{B,A}$. Our Theorem~\ref{main} improves this to $O(n^2m^{3.5} 4^m)$. Similarly, their upper bound on computing a global minimum was $O(n^3m^{7.5}(e\ln m+e)^m)$, which Theorem~\ref{main} improves to $O(n^3m^{7.5} 4^m)$.

These bounds are still not polynomial in $m$, but unfortunately we cannot even hope to achieve a polynomial bound on these problems only by improving bounds on $f(m)$ as it was shown in~\cite{BHJKRST} that $f(m)= \Omega(2^m/\sqrt{m})$.

\section{Proofs}
We start with a lemma about reachable sets. We think of the $m$ buyers' preference lists as a \emph{preference matrix} $M$ of $m$ infinite rows. The $(i,j)$-entry of $M$ is the index of the house that buyer $i$ has in the $j$-th place in his preference list.

\begin{lemma}\label{ellem}
Let $\ell=\ell(m)$ be the largest integer with $\ell^{\ell}\le m$. Then for any list of preferences of $m$ buyers, we either have $|\bigcap_{\tau}s(\tau)|\ge \ell-4$ or there is a set $X$ of houses with $|X|\ge \ell^2$ such that for any POM $\tau$ we have $|s(\tau)\cap X|\ge |X|-\ell$.
\end{lemma}

\begin{proof}
If we can find $\ell$ rows such that the first $\ell-4$ columns of the preference matrix are constant restricted to these rows, then obviously the elements in these rows belong to every $s(\tau)$. We are going to work in rounds, and we want to find $A_1 \supseteq A_2 \supseteq \dots \supseteq A_{\ell-4}$ nested sets of rows such that $|A_i|\ge \frac{m}{\ell^i}$ and the first $i$ columns of the preference matrix are constant on the rows in $A_i$. If we succeed, then we have $|A_{\ell-4}|\ge \ell^4$. It is easy to see that every element from the first $\ell-4$ columns in the rows of $A_{\ell-4}$ must be in every $s(\tau)$, thus $|\bigcap_{\tau}s(\tau)|\ge \ell-4$.

Clearly, every element in the first column belongs to all sets $s(\tau)$, therefore either there are at least $\ell-4$ elements in the first column and then we are done, or there is an element that appears at least $\frac{m}{\ell-4}\ge \frac{m}{\ell}$ times. These rows form $A_1$.

%I don't understand the sentence below starting with "Also" -ct
Suppose we have defined $A_i$ such that the first $i$ columns are constant on the rows belonging to $A_i$. If we are not able to form $A_{i+1}$, then every element appears less than $\frac{m}{\ell^{i+1}}$ times in the $(i+1)$st column on the rows of $A_i$. Also, any element $x$ that appears at least $i+1$ times in column $i+1$ on the rows of $A_i$ must belong to every $s(\tau)$. Indeed, %at least $\ell+1$ 
the $i+1$ buyers corresponding to these rows share their first $i$ preferences, so one of them will have to pick its $(i+1)$-st preferred house $x$. Either that buyer or a buyer from a different row must take $x$.

Let $X$ be the set of elements appearing in the $(i+1)$-st column in the at least $\frac{m}{\ell^{i}}$ rows belonging to $A_i$. By the argument above, if $|\cap_\tau s(\tau)|<\ell-4$ and $A_{i+1}$ does not exist with the required properties, then each element of $X$ appears less than $\frac{m}{\ell^{i+1}}$ times, and all but at most $\ell-5$ elements appear at most $i$ times in the $(i+1)$-st column in those rows. Hence

%Therefore; if $|\cap_\tau s(\tau)|<\ell-4$ and $A_{i+1}$ does not exist with the required properties, then for the set $X$ of elements appearing in the $(i+1)$st column in the $\frac{m}{\ell^{i}}$ rows belonging $A_i$, we have that each element appears less than $\frac{m}{\ell^{i+1}}$ times, and all but at most $l-5$ elements appear at most $i$ times. Hence

\[
(\ell-5)\frac{m}{\ell^{i+1}}+(|X|-\ell+5)i\ge \frac{m}{\ell^i}.
\]
Rearranging and using that $i\le \ell-4$ yields $|X|\ge \frac{m}{\ell^{i+2}}\ge \ell^2$. Observe that for any $\tau$ the set $s(\tau)$ contains at least $|X|-i$ elements of $X$, those that belong to rows of $A_i$ of which the corresponding buyer has not picked their house from its first $i$ preferred houses.
\end{proof}

We remark that it is easy to obtain some negligible improvement by a more precise analysis of the above arguments. For example, we used that an element that appears at least $i+1$ times in column $i+1$ on the rows of $A_i$ must belong to every $s(\tau)$. However, in addition to these elements, the $i$ elements appearing in the first $i$ columns in those rows also belong to every $s(\tau)$. For simplicity, we did not take this (and similar arguments) into account in our calculation.

Let us now introduce the last and most important notion and result that we need for the proof of Theorem~\ref{main}. We say that a family of $k$ sets $A_1,\dots,A_k$ is disjointly representable, if there are elements $x_1,\dots,x_k$ such that $x_i\in A_j$ if and only if $i=j$ (for all $1\le i,j\le k$). In other words, none of the $A_i$'s are contained in the union of the others. We will use the following theorem of Frankl and Pach \cite{FP84}.

\begin{theorem}[Frankl, Pach \cite{FP84}]
\label{Frankl}
Let $f(r,k)$ be the maximum size of an $r$-uniform family that does not contain $k$ disjointly representable sets. Then we have
\[
f(r,k) \le \binom{r+k-1}{k-1}.
\]
\end{theorem}

% Sets of possibly selected houses sounds terrible, but I'm searching for better phrasing.  
\begin{proof}  [Proof of Theorem \ref{main}]
A bound of $f(m)\le g(m)\le \binom{2m}{m}$ is immediate. Indeed, using only $k=m+1$ in the definition of property $P$, we obtain that any $m+1$ sets in $\cF$ are not disjointly representable. Applying Theorem~\ref{Frankl} with $k=m+1$, $r=m$ yields $g(m)\le \binom{2m}{m}$.

To obtain the improvement of (i), observe that if $s(\tau_1),s(\tau_2),\dots$ are the sets of possibly selected houses, then the family $s(\tau_1)\setminus (\bigcap_\tau s(\tau))$, $s(\tau_2)\setminus (\bigcap_\tau s(\tau))$,\dots still does not contain $m+1$ disjointly representable sets. If Lemma~\ref{ellem} yields $|\bigcap_\tau(s(\tau)|\ge \ell-4$, then we can apply Theorem~\ref{Frankl} with $k=m+1$, $r=m-\ell+4$ to get $f(m)\le \binom{2m-\ell+4}{m}$. Otherwise Lemma \ref{ellem} yields a set $X$ of at least $\ell^2$ elements such that for every POM $\tau$ we have $|s(\tau)\cap X|\ge |X|-\ell$. Let $\cF=\{s(\tau):\tau ~\text{is a POM}\}$ and for any set $Y$ let $\cF_Y=\{F\setminus Y:Y\subset F\in \cF\}$. By the pigeonhole principle, there exists a subset $Y\subset X$ of size $|X|-\ell$ such that $|\cF|\le \frac{|\cF_Y|}{\binom{|X|}{\ell}}$. Observe that $\cF_Y$ also has property $P$, thus $\cF_Y$ cannot contain $m+1$ disjointly representable sets. Therefore, by Theorem~\ref{Frankl} we obtain $|\cF_Y|\le \binom{2m-|X|+\ell}{m}$ and thus $|\cF|\le \binom{|X|}{\ell}\binom{2m-|X|+\ell}{m}$. As $|X|\ge \ell^2$, the latter expression is smaller than $\binom{2m-\ell+4}{m}$ if $m$ is large enough.

To see the improvement of (ii), we will use only that property $P$ holds for $k=2$ and $k=m+1$. As we have seen, $k=m+1$ means there are no $m+1$ disjointly representable sets in $\cF$. On the other hand, $k=2$ means that for every pair of sets, one has at most $\lfloor m/2\rfloor$ elements not in the other, which means their intersection has size at least $\lceil m/2\rceil$. We will only use that the intersection is non-empty. As observed by Frankl and Pach~\cite{FP84}, if $F_1,F_2,\dots,F_t$ do not contain $m+1$ disjointly representable sets, then for any $1\le i\le t$ a minimal set $E_i \subset (\cup_jF_j)\setminus F_i$ that meets every $F_j$ ($j\neq i$) has size at most $m$. Then $F_i \cap E_{i'}\neq \emptyset$ if and only if $i\neq i'$. For pairs with this property, Bollob\'as's famous inequality states that $\sum_{j=1}^t\frac{1}{\binom{|F_j|+|E_j|}{|E_j|}}\le 1$, in particular $t\le \binom{2m}{m}$ as all $F_j's$ and $E_j$'s have size at most $m$.

 Since the $F_j$'s are intersecting, one can define another collection $\{(A_i,B_i)\}_{i=1}^{2t}$ with $A_i=F_i$, $B_i=E_i$ for $1\le i \le t$ and $B_i=F_{2t-i+1}, A_i=E_{2t-i+1}$ for $t+1\le i\le 2t$. Then this collection is \textit{skew cross-intersecting}, i.e., for any $i<j$ we have $A_i\cap B_j\neq \emptyset$ and $A_i\cap B_i=\emptyset$ for all $1\le i \le 2t$. As proved by Frankl~\cite{F} and Kalai~\cite{K}, if in such collections all sets have size $m$, then the number of pairs is at most $\binom{2m}{m}$. Thus $|\cF|=t\le \frac{1}{2}\binom{2m}{m}=\binom{2m-1}{m}$, as claimed.
\end{proof}

%Taking $k=m+1$ in Theorem~\ref{reach} implies that among any $k$ reachable sets one is contained in the union of the others.  That is, any set of $k$ reachable sets is not distinctly representable.  Applying Theorem~\ref{Frankl} with $k=m+1$, $r=m$ we obtain the following.

%This also improves something from \cite{HJK}.

Note that Theorem \ref{Frankl} is very close to being sharp (also shown in~\cite{FP84}). We believe that both bounds in Theorem \ref{main} are far from being sharp. The best construction we are aware of is summarized in the following simple proposition.
%The best lower bound we have is $\ge\binom{\lfloor 3m/2\rfloor}{m}$, as shown by the family $\binom{[\lfloor 3m/2\rfloor]}{m}$.

\begin{proposition}
$h(m)\ge\binom{\lfloor 3m/2\rfloor}{m}$.
\end{proposition}

\begin{proof}
Consider the family $\binom{[\lfloor 3m/2\rfloor]}{m}$. Property $Q$ holds for this family for $k=1$ because of the uniformity, and for larger $k$ because $\sum_{i=1}^k |F_i|$ cannot be larger than the underlying set.
\end{proof}

Many variants, generalizations and strengthenings~\cite{F,Fu,K,KKK,KNPV,L,P,SW,T,Tu1,Tu2,Tu3} of Bollob\'as's inequality have been established. By the remarks above, to improve our bounds, we would be interested in versions of Bollob\'as's inequality when the $A_i$'s satisfy some extra intersection property. Let $j(m)$ denote the largest cardinality of an $(\lceil m/2\rceil)$-intersecting family without $m+1$ disjointly representable sets. Then by the arguments above we have $g(m)\le j(m)\le\binom{2m-1}{m}$. 

We finish our note with two conjectures starting with the 1-intersecting case. 

%I don't think this is well known -CT

%\begin{rem} One idea was to use that Theorem~\ref{reach} with $k=2$ implies that the family of reachable sets of houses is $m/2$-intersecting.  Then using that Theorem~\ref{Frankl} is immediate from Bollob\'as's two family theorem we would try to prove a $t$-intersecting version of this theorem. \end{rem}

\begin{conjecture}\label{tuz}%[We think first proposed by Tuza]
Let $\{(A_i,B_i)\}_{i=1}^m$ be a collection of pairs of sets such that $A_i \cap B_j = \varnothing$ if and only if $i=j$ and $A_i \cap A_j \neq \varnothing$ for all $i$ and $j$.  Let $\abs{A_i}=a$ and $\abs{B_i}=b$ and $a \le b$, then
\[
m \le \binom{a+b-1}{a-1}
\]
\end{conjecture}

By the reasoning in the proof of Theorem \ref{main} (ii) (the idea of which goes back to Poljak and Tuza~\cite{PT}), Conjecture \ref{tuz} is true if $a=b$. Moreover, as Scott and Wilmer~\cite{SW} have recently established a Bollob\'as type inequality for skew intersecting set pairs $\{A_i,B_i\}_{i=1}^m$ with $|A_1|\le |A_2|\le \dots \le |A_m|$ and $|B_1|\ge |B_2|\ge \dots \ge |B_m|$, their result implies that in Conjecture \ref{tuz} the bound $m\le \frac{1}{2}\binom{a+b}{a}$ holds.

%Indeed, given pairs $\{(A_i,B_i)\}_{i=1}^m$ with the above property, one can define another collection $\{(F_i,G_i)\}_{i=1}^{2m}$ with $F_i=A_i,G_i=B_i$ for $1\le i \le m$ and $F_i=B_{2m-i+1}, G_i=A_{2m-i+1}$ for $m+1\le i\le 2m$. Then this collection is \textit{skew cross-intersecting}, i.e. for any $i<j$ we have $F_i\cap G_j\neq \emptyset$. As proved by Frankl \cite{F} and Kalai \cite{K} if in such collections all sets have size $k$, then the number of pairs is at most $\binom{2k}{k}$. Therefore if $a=b$, then $2m\le \binom{2a}{a}=2\binom{2a-1}{a-1}$. 

We propose a straightforward generalization for the $t$-intersecting case. Let $AK(n,k,t)$ denote the maximum size of a $k$-uniform $t$-intersecting family $\cF\subseteq \binom{[n]}{k}$. The value of $AK(n,k,t)$ for any $n,k,t$ was determined by Ahlswede and Khatchatrian~\cite{AK}.

\begin{conjecture}\label{ak}
Let $\{(A_i,B_i)\}_{i=1}^m$ be a collection of pairs of sets such that $A_i \cap B_j = \varnothing$ if and only if $i=j$ and $\abs{A_i \cap A_j} \ge t$ for all $i$ and $j$.  Let $\abs{A_i}=a$ and $\abs{B_i}=b$ and $a \le b$, then
\[
m \le AK(a+b,a,t).
\]
\end{conjecture}

Conjecture \ref{ak}, if true, would yield an upper bound $f(m)\le j(m)\le \max\{|\cF_i|:i=0,1,\dots,m/2\}$, where $\cF_i=\{F\in \binom{[2m]}{m}:|F\cap [\lceil \frac{m}{2}\rceil+2i]|\ge \lceil \frac{m}{2}\rceil+i\}$. Note that the maximum is taken approximately at $i=\frac{m}{8}$.

\vskip 1truecm

\textbf{Acknowledgements.} The results presented in this note were obtained during the 9th Eml\'ekt\'abla Workshop, in G\'ardony, June 2019. The research was supported by the National Research, Development and Innovation Office NKFIH  under  the  grants  K  116769,  KH130371  and  SNN  129364.   The research  of  Gerbner was supported by the J\'anos Bolyai Research Fellowship. The research of Keszegh was supported by the Lend\"ulet program of the Hungarian Academy of Sciences (MTA), under the grant LP2017-19/2017. The research of Methuku and Tompkins was supported by IBS-R029-C1.

\end{document}